\documentclass[runningheads]{llncs}
\usepackage[T1]{fontenc}
\usepackage{graphicx}
\usepackage{amsmath,amssymb,mathtools,bm}
\usepackage{algorithm}
\usepackage{algpseudocode}

\DeclareMathOperator*{\argmax}{arg\,max}
\newcommand{\R}{\mathbb{R}}
\newcommand{\E}{\mathbb{E}}
\newcommand{\Prob}{\mathbb{P}}
\newcommand{\norm}[1]{\left\lVert #1 \right\rVert}

\newcommand{\one}{\bm{1}}
\spnewtheorem{assumption}{Assumption}{\bfseries}{\itshape}
\begin{document}
\title{A Parameter-Free  Zeroth-Order Algorithm for Decentralized Stochastic Convex Optimization
}
\titlerunning{Parameter-Free Decentralized Zeroth-Order Optimization}
%
\author{Jiawei Chen\inst{1,2} \and
Alexander Rogozin\inst{1}}
\authorrunning{J. Chen, A.V. Rogozin}
%
\institute{MIRIAI, Moscow, Russia \\\email{chen.j@miriai.org}\and
Appied AI Institute, Moscow, Russia
}
\maketitle              
\begin{abstract}
We consider decentralized stochastic convex optimization on connected network, 
in which  gradients of agents are unavailable and each agent can query only noisy
function values of its own local objective. The goal is to minimize the average
objective over a compact convex domain using only local two point zeroth-order
oracles and peer-to-peer communication. We propose a decentralized POEM
method (D-POEM) that combines symmetric two point smoothing with adaptive
radius and stepsize rules, thereby avoiding prior knowledge of the Lipschitz
constant and diameter. For convex Lipschitz continuous objectives, we prove an
upper bound  that separates a centralized
optimization term from a network disagreement term.
We further conduct the numerical experiments to demonstrate POEM outperforms existing distributed zeroth-order method.
\keywords{Gradient free optimization  \and Parameter free method \and Decentralized convex optimization.}
\end{abstract}
\section{Introduction}
Decentralized zeroth-order optimization studies multi-agent problems in which
each node can query only noisy function values of its local objective and can
communicate only with graph neighbors. This setting appears in simulation-based
control, bandit feedback, experimental design, and privacy-constrained
learning, and it is challenging because the agents must optimize
collaboratively without direct gradient information.

Classical stochastic zeroth-order methods based on randomized smoothing and
finite differences show that convex optimization from function values is
possible with near-optimal dimension dependence, especially with two-point
feedback \cite{duchi2015optimal,nesterov2017random,shamir2017optimal}. Their performance, however,
typically depends on tuned stepsizes and smoothing radii. In the centralized
setting, POEM replaces manual tuning parameter with an adaptive schedule for stepsize and 
smoothing radius \cite{renparameter}. It is closely related to the first
order DoG principle \cite{ivgi2023dog}, where the stepsize is adapted from traveled
distance and accumulated gradient norms.

Extending this idea to decentralized zeroth-order optimization is nontrivial.
The adaptive quantities used by POEM are global, whereas each agent observes
only local function values and local iterates. Existing decentralized
zeroth-order methods study time-varying networks, gradient tracking, variance
reduction, constrained updates, and communication efficiency
\cite{hou2024distributed,li2021distributed,li2024achieving,lin2024decentralized,mhanna2023single,mu2024}, but they typically
rely on preset schedules or target different regimes such as nonconvex or
federated optimization. Recent first-order work achieves parameter-free
decentralized adaptation \cite{kuruzov2024achieving}, but it does not address the
zeroth-order coupling among smoothing, estimator variance, and communication.
To the best of our knowledge, this is the first work  to solve stochastic 
convex decentralized zeroth-order optimization with fully
adaptive stepsizes and smoothing, without prior Lipschitz or horizon knowledge. 

The main contributions of this work are the following.
\begin{itemize}
\item We propose D-POEM, a decentralized parameter-free stochastic zeroth-order
    method based on two-point zeroth order oracle and consensus-based projected 
    updates, together with a fully decentralized adaptive scheduler that updates 
    the iterate and a radius proxy using only one consensus per iteration.
\item We establish an upper bound on the function gap  explicitly separating the optimization 
        error from the consensus error, thereby revealing how network 
        disagreement influences convergence.
\item We conduct a practical experiments to demonstrate the performance of D-POEM
    outperforms existing distributed zeroth order method.
\end{itemize}
\section{Problem Setup}

We study the decentralized stochastic convex problem
\begin{equation}
\label{eq:problem}
\min_{x \in \mathcal{X}} f(x)
:=
\frac{1}{n}\sum_{i=1}^n f_i(x),
\qquad
f_i(x) := \E_{\xi_i}\big[F_i(x;\xi_i)\big].
\end{equation}
Each agent $i \in \{1,\dots,n\}$ knows only its own stochastic oracle $F_i$,
can query only function values, and can communicate only with its graph neighbors.
Let \(x^\star\in\arg\min_{x\in \mathcal{X}} f(x)\) denote a minimizer of \eqref{eq:problem}.
All oracle randomness, including \(\xi_i\) and the random directions used by
the algorithm, is defined on a common probability space
\((\Omega,\mathcal{F},\Prob)\). Expectations and probabilities are taken with
respect to this space unless stated otherwise.

Let $W = (w_{ij}) \in \mathbb{R}^{n \times n}$ be a symmetric, doubly stochastic
mixing matrix adapted to the communication graph. Write
\(
\sigma := \norm{W-J}_2 < 1.
\)
where \(J := \frac{1}{n}\one\one^\top\) is the averaging matrix. The matrix
\(J\) averages over the network, while \(\sigma\) measures how fast contract 
of disagreement via one communication step. Smaller \(\sigma\) means a
better connected graph.

\begin{assumption}
\label{ass:domain}
The set $\mathcal{X} \subset \R^d$ is compact and convex with diameter
\(
D_{\mathcal{X}} := \max_{x,y \in \mathcal{X}} \norm{x-y}.
\)
\end{assumption}

\begin{assumption}
\label{ass:lipschitz}
There exists a convex set \(\mathcal{X}^+\supseteq \mathcal{X}\) that contains all oracle query
points used by the algorithm. For every realization of \(\xi_i\), the map
\(x \mapsto F_i(x;\xi_i)\) is convex and \(L\)-Lipschitz on \(\mathcal{X}^+\).
\end{assumption}

\begin{assumption}
\label{ass:oracle}
For every $x \in \mathcal{X}$, $\mu > 0$, and $u \sim \mathrm{Unif}(\mathbb{S}^{d-1})$, agent $i$ can query the pair
\[
F_i(x+\mu u;\xi_i),
\qquad
F_i(x-\mu u;\xi_i),
\]
using the same sample $\xi_i$.
\end{assumption}

Assumption~\ref{ass:domain} prevents the iterates from escaping to infinity. 
Assumption~\ref{ass:lipschitz} controls how fast each oracle can change. 
Assumption~\ref{ass:oracle} is the standard zeroth-order requirement that 
lets us build a symmetric two-point gradient estimator with reduced noise.

\begin{algorithm}[t]
\caption{Decentralized POEM (D-POEM)}
\label{alg:dpoem}
\begin{algorithmic}[1]
\Require Number of agents $n$, mixing matrix $W=[w_{ij}]$, initial point $x_{i,0}=x_0\in\mathcal{X}$ for all $i\in[n]$, radius floor $r_\epsilon\in(0,D_{\mathcal{X}}]$, horizon $T\ge 1$
\State $\bar r_{i,-1}\gets r_\epsilon,\quad G_{i,-1}\gets r_\epsilon^2,\qquad \forall i\in[n]$
\For{$t=0,1,\dots,T-1$}

    \ForAll{agents $i\in[n]$ \textbf{in parallel}}
        \State $\hat r_{i,t} = \max\{\bar r_{i,t-1},\|x_{i,t} - x_{i,0} \| \}$ \label{algo:r_update}
        \State Receive $\{\hat r_{j,t}\}_{j\in\mathcal N_i}$ from neighbors
        \State $\bar r_{i,t} \gets \sum_{j=1}^n w_{ij}\hat r_{j,t}$ \label{algo:r_mix}
        \State $\mu_{i,t} \gets \bar r_{i,t} \sqrt{\dfrac{d}{t+1}}$ \label{algo:mu_update}
        \State Sample $v_{i,t}\sim \mathrm{Unif}(\mathbb S^{d-1})$ and a fresh oracle sample $\xi_{i,t}$
        \State  $
        g_{i,t}\gets
        \dfrac{d}{2\mu_{i,t}}
        \Bigl(
        F_i(x_{i,t}+\mu_{i,t} v_{i,t};\xi_{i,t})
        -
        F_i(x_{i,t}-\mu_{i,t} v_{i,t};\xi_{i,t})
        \Bigr)v_{i,t}
        $ \label{algo:g_update}
        \State $G_{i,t}\gets G_{i,t-1}+\|g_{i,t}\|^2$ \label{algo:G_update}
        \State $\eta_{i,t}\gets \dfrac{\bar r_{i,t}}{\sqrt{G_{i,t}}}$   \label{algo:eta_update}  
        \State Receive $\{x_{j,t}\}_{j\in\mathcal N_i}$ from neighbors
        \State $z_{i,t}\gets \sum_{j=1}^n w_{ij}x_{j,t}$ \label{algo:z_update}
        \State  $
        x_{i,t+1}\gets
        \Pi_{\mathcal{X}}\!\left(
            z_{i,t}
            -
            \eta_{i,t}g_{i,t}
        \right)
        $ \label{algo:x_update}
    \EndFor
\EndFor
\State \textbf{Output:} iterate sequence $\{x_{i,T}\}_{i=1}^n$
\end{algorithmic}
\end{algorithm}

\section{Method}

We now describe D-POEM, our decentralized adaptation of POEM for stochastic
zeroth-order optimization over a communication graph. Let \(\Pi_{\mathcal{X}}\)
denote Euclidean projection onto \(\mathcal{X}\). Algorithm~\ref{alg:dpoem}
summarizes one round of D-POEM. Agent \(i\) first updates a local radius proxy
and averages it with its neighbors (Lines~\ref{algo:r_update}--\ref{algo:r_mix}).
The mixed radius \(\bar r_{i,t}\) then defines the smoothing parameter
\(\mu_{i,t}\) and the adaptive stepsize \(\eta_{i,t}\) through the accumulator
\(G_{i,t}=r_\epsilon^2+\sum_{s=0}^t \norm{g_{i,s}}^2\)
(Lines~\ref{algo:mu_update}, \ref{algo:G_update}, and
\ref{algo:eta_update}). Next, the agent forms a symmetric two-point estimator
(Line~\ref{algo:g_update}), computes the mixed iterate \(z_{i,t}\)
(Line~\ref{algo:z_update}), and performs the consensus-plus-projection update
(Line~\ref{algo:x_update}). Each round therefore uses two local function
evaluations, one scalar mixing average step for the radius proxy, and one
vector mixing average step for the primal iterate.

For the analysis, we use
\[
\bar x_t:=\frac{1}{n}\sum_{i=1}^n x_{i,t},
\qquad
\bar r_t:=\frac{1}{n}\sum_{i=1}^n \bar r_{i,t},
\qquad
\tilde{x}_t
:=
\frac{\sum_{k=0}^{t-1}\bar r_k\,\bar x_k}{R_t},
\]
where \(R_t:=\sum_{k=0}^{t-1}\bar r_k\). The iterate \(\tilde x_t\) is the
natural decentralized analogue of the averaged output used in the centralized
POEM analyses. The weights \(\bar r_k\) are not arbitrary: they are the same
adaptive radii that govern the smoothing and stepsize rules, and this
alignment is what makes the later weighted-regret argument works well.

\section{Main Result}
All proofs are deferred to Appendix~\ref{app:proofs}.
Our analysis considers the weighted average of the 
sequence $\tilde{x}_t$ generated from Algorithm \ref{alg:dpoem}. 
Since the objective $f_{i}$ is convex and with Jensen inequality 
we can get following decomposition, for every \(t\ge 1\):
\begin{equation}
\label{eq:main-decomposition}
R_t\bigl(f(\tilde x_t)-f(x^\star)\bigr)
\le
\mathcal W_t+\mathcal N_t+\mathcal B_t+\mathcal E_t,
\end{equation}
where 
\begin{align*}
\mathcal W_t
&:=
\sum_{k=0}^{t-1}\frac{\bar r_k}{n}\sum_{i=1}^{n}\langle g_{i,k},z_{i,k}-x^\star\rangle,
\mathcal N_t
:=
\sum_{k=0}^{t-1}\bar r_k\cdot \frac{1}{n}\sum_{i=1}^{n}
\langle \Delta_{i,k},x_{i,k}-x^\star\rangle,\\
\mathcal B_t
&:=
\sum_{k=0}^{t-1}\frac{2L\bar r_k}{n}\sum_{i=1}^{n}\mu_{i,k},\\
\mathcal E_t
&:=
\sum_{k=0}^{t-1}\frac{\bar r_k}{n}\sum_{i=1}^{n}\langle g_{i,k},x_{i,k}-z_{i,k}\rangle
+
\sum_{k=0}^{t-1}\frac{L\bar r_k}{n}\sum_{i=1}^{n}\|x_{i,k}-\bar x_k\|.
\end{align*}
write \(\Delta_{i,k}:=\nabla f_{i,\mu_{i,k}}(x_{i,k})-g_{i,k}\),
\(z_{i,k}:=\sum_{j=1}^n w_{ij}x_{j,k}\), and
\(\hat G_t:=\max_{i\in[n]} G_{i,t}\).The first three terms are 
from the centralized POEM errors, while \(\mathcal E_t\) is 
the disagreement penalty. The next four lemmas control
these terms separately.

We begin with the term \(\mathcal W_t\), which is the projected gradient 
descent component inherited from the projected POEM recursion. This part 
need to deal with the mismatch of local radii and global radius, here we 
show that it is still bounded with one consensus round on the local radii $\hat r_{i,k}$.
\begin{lemma}[Weighted regret bound]\label{lem:weighted-regret}
For every \(t\ge 1\),
\[
\mathcal W_t
\le
3D_{\mathcal{X}}\bar r_t\sqrt{\hat G_{t-1}}.
\]
\end{lemma}

Once the descent term is under control, the next term comes from the
fact that the algorithm only  get zeroth-order gradient estimator. The
following lemma isolates this stochastic noise from gradient as 
a martingale difference sequence.

\begin{lemma}[Estimator noise]\label{lem:dec-noise}
For every \(\delta\in(0,1)\), there exists an event \(\Omega_{T,\delta}\subseteq \Omega\)
with \(\Prob(\Omega_{T,\delta})\ge 1-\delta\) such that, simultaneously for all
\(1\le t\le T\),
\[
\mathcal N_t
\le
8D_{\mathcal{X}}\bar r_{t-1}
\sqrt{\theta_{t,\delta}\hat G_{t-1}+4L^2d^2\theta_{t,\delta}^2}.
\]
where \(\theta_{t,\delta}:=\log\!\bigl(60\log(6t/\delta)\bigr)\).
\end{lemma}

Since D-POEM works with smoothed local objectives, we next bound the bias
introduced by the smoothing radii \(\mu_{i,k}\).

\begin{lemma}[Smoothing bias]\label{lem:smoothing-bias}
For every \(t\ge 1\),
\(
\mathcal B_t
\le
4D_{\mathcal{X}}L\bar r_t\sqrt{dt}.
\)
\end{lemma}

At this stage, the optimization errors from POEM is bounded.
What remains is the error from the mismatch
between local iterates and their network average. We handle this in two steps:
first we estimate the average size of the disagreement process itself, and
then we convert that estimate into a bound on \(\mathcal E_t\).

\begin{lemma}[Time-averaged consensus error]\label{lem:avg-consensus-A}
Under Assumptions~\ref{ass:domain}--\ref{ass:oracle} and common initialization
\(x_{1,0}=\cdots=x_{n,0}\), for any horizon \(K\ge 1\),
\[
\frac{1}{K}\sum_{k=0}^{K-1}\|\widetilde X_k\|_F
\le
\frac{D_{\mathcal{X}}}{1-\sigma}
\sqrt{\frac{1}{K}\sum_{i=1}^{n}\log\!\Bigl(\frac{G_{i,K-1}}{G_{i,0}}\Bigr)}.
\]
where \(X_k\in\mathbb R^{n\times d}\) stacks the iterates \(x_{i,k}\) as rows
and \(\widetilde X_k:=(I-J)X_k\).
\end{lemma}

The previous lemma quantifies how much disagreement can accumulate on
average along the trajectory. The next bound feeds that control back into the
decomposition and turns it into an explicit estimate for the disagreement
penalty \(\mathcal E_t\).

\begin{lemma}[Consensus error]\label{lem:consensus-exp}
For every \(t\ge 1\),
\[
\mathcal E_t
\le
\frac{2}{n}\sum_{k=0}^{t-1}\bar r_k\,\|\mathbf G_k\|_F\,E_k
+
\frac{L}{\sqrt n}\sum_{k=0}^{t-1}\bar r_k E_k.
\]
where \(E_k:=\|\widetilde X_k\|_F\) and \(\mathbf G_k\in\mathbb R^{n\times d}\)
stacks the row vectors \(g_{i,k}^\top\).
\end{lemma}

Combining the preceding lemmas with the decomposition
\eqref{eq:main-decomposition} yields the main finite-time guarantee. The
result keeps the network term explicit, which is useful before passing to the
expected-rate statement.

\begin{theorem}[Function gap]\label{thm:function-gap}
Under Assumptions~\ref{ass:domain}--\ref{ass:oracle} and the hypotheses of
Lemma~\ref{lem:dec-noise}, on the event \(\Omega_{T,\delta}\) from
Lemma~\ref{lem:dec-noise}, for every \(1\le t\le T\),
\begin{equation}\label{eq:function-gap-divided}
f(\tilde x_t)-f(x^\star)
\le
\frac{16D_{\mathcal{X}}\theta_{T,\delta}}{R_t/\bar r_t}
\bigl(\sqrt{\hat G_T}+Ld+L\sqrt{dT}\bigr)
\;+\;
\frac{\mathcal E_t}{R_t}.
\end{equation}
\end{theorem}

The final step is to show that, after choosing the most favorable weighted
output time and taking conditional expectation, the remaining disagreement
term contributes only the standard network factor \(1/(1-\sigma)\).

\begin{theorem}[Conditional convergence rate]\label{thm:expected-rate}
Under the assumptions of Theorem~\ref{thm:function-gap} and common
initialization \(x_{1,0}=\cdots=x_{n,0}\),
Then
\begin{equation}
\label{eq:expected-rate-tilde}
\mathbb E\!\left[f(\bar x_{\mathrm{out}})-f(x^\star)\mid \Omega_{T,\delta}\right]
=
\widetilde{\mathcal O}\!\left(
\frac{L D_{\mathcal{X}}\sqrt d}{\sqrt T}
\;+\;
\frac{L D_{\mathcal{X}}\sqrt d}{(1-\sigma)\sqrt T}
\right),
\end{equation}
where
\[
\tau_T:=\argmax_{1\le t\le T}\frac{R_t}{\bar r_t},
\qquad
\bar x_{\mathrm{out}}:=\tilde x_{\tau_T},
\]
and \(\Prob(\Omega_{T,\delta})\ge 1-\delta\).
\end{theorem}

For this choice of \(\tau_T\), the monotone-sequence bound used in the POEM
analysis yields
\[
\frac{R_{\tau_T}}{\bar r_{\tau_T}}
\ge
\frac{T}{e\bigl(1+\log(D_{\mathcal{X}}/r_\epsilon)\bigr)}.
\]
\begin{remark}
As $\delta$ approaches $0$, we have $\Omega_{T,\delta} \to \Omega_{T,0}$, which implies
\[
\mathbb{E}\!\left[f(\bar{x}_T)-f(x_\star)\mid \Omega_{T,\delta}\right]
\to
\mathbb{E}\!\left[f(x_T)-f(x_\star)\mid \Omega_{T,0}\right]
=
f(x_T)-f(x_\star).
\]
\end{remark}
\begin{corollary}[Stochastic zeroth-order and communication complexity]\label{cor:expected-complexity}
Under the assumptions of Theorem~\ref{thm:expected-rate}, to guarantee
\(
\mathbb E\!\left[f(\bar x_{\mathrm{out}})-f(x^\star)\mid \Omega_{T,\delta}\right]
\le \varepsilon,
\)
it suffices to take
\begin{equation}
\label{eq:T-eps-comm}
T_\varepsilon
=
\widetilde{\mathcal O}\!\left(
\frac{L^2 D_{\mathcal{X}}^2 d}{\varepsilon^2}
\left(
1+\frac{1}{(1-\sigma)^2}
\right)
\right).
\end{equation}
Since each iteration of Algorithm~\ref{alg:dpoem} query zeroth order oracle two times and two communication rounds, the per-agent stochastic zeroth-order query
complexity,and the communication complexity are
\[
N_{\mathrm{comm}}(T_\varepsilon)=N_{\mathrm{ZO}}^{(i)}(T_\varepsilon)
=
\widetilde{\mathcal O}\!\left(
\frac{L^2 D_{\mathcal{X}}^2 d}{\varepsilon^2}
\left(
1+\frac{1}{(1-\sigma)^2}
\right)
\right),
\nonumber
\label{eq:complexity-agent}
\]
\end{corollary}

\section{Numerical Experiments}
We evaluate D-POEM on decentralized stochastic hinge-loss binary classification:
\[
\min_{x\in\mathcal{X}} f(x)\coloneq \frac{1}{n}\sum_{i=1}^{n}\mathbb{E}_{(a,b)}\!\left[f_i(x;a,b)\right],
\]
where $f_i(x;a,b)=\max\{0,\,1-b\langle a,x\rangle\}$, with \((a,b)\in\mathbb{R}^d\times\{\pm1\}\) 
sampled uniformly from the local  dataset of agent \(i\). We conduct experiments 
on the \texttt{mushrooms} \((d=112,\, N=8124)\), \texttt{a9a} \((d=123,\, N=32{,}561)\), 
and \texttt{w8a} \((d=300,\, N=49{,}749)\) datasets. The data are partitioned i.i.d.\ across the 20 agents. The communication network 
is modeled as an Erd\H{o}s--R\'enyi graph with edge probability \(0.25\), 
and the feasible set \(\mathcal{X}\) is the Euclidean ball of radius \(1\).
In our experiments D-POEM  with $r_{\epsilon} = 0.1$ compares against two variants of the distributed subgradient-free algorithm(DSF)\cite{wang2019distributed}: 
the tuned version \texttt{DSF-T} and the default-setting version \texttt{DSF-D}.

\begin{figure}[t]
\centering
\includegraphics[width=\textwidth]{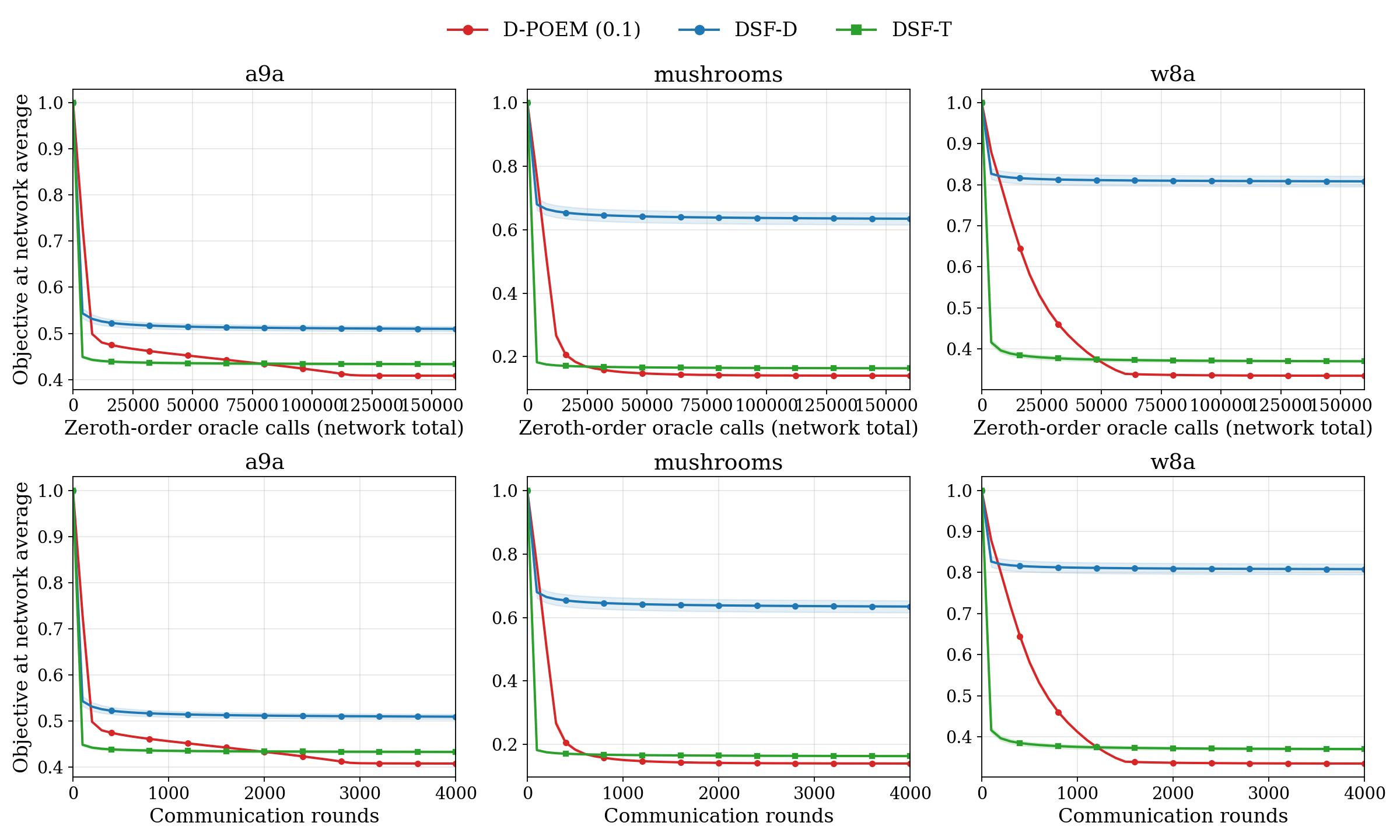}
\caption{Twenty-agent  benchmark on \texttt{mushrooms}, \texttt{a9a},
and \texttt{w8a}. The top row plots objective at the network average
versus total network-wide zeroth-order oracle calls; the bottom row plots
the same objective versus communication rounds.}
\label{fig:numerical-20agents}
\end{figure}

Figure~\ref{fig:numerical-20agents} shows that the DSF methods converge
faster in the early stage. However, \texttt{DSF-D} performs noticeably
worse than D-POEM. As the iterations 
continue, D-POEM eventually achieves better performance than 
\texttt{DSF-T}. This advantage is substantial on the 
more challenging \texttt{w8a} task. Meanwhile, \texttt{DSF-T} still 
requires careful tuning of the step size and perturbation scale. 
These results indicate that D-POEM is a robust method that performs
well without parameter tuning.

\section{Conclusion and Discussion}
We studied decentralized stochastic convex optimization with zeroth-order
oracle over a connected network and proposed D-POEM, a parameter-free
distributed method that combines symmetric two-point gradient estimation with
adaptive smoothing radii and stepsizes. The method requires only local
function evaluations together with one scalar consensus step for the radius
proxy and one vector consensus step for the primal iterate at each round. For
convex Lipschitz objectives on a compact domain, we established a finite-time
bound that separates the optimization term from the network disagreement term,
and derived the conditional rate
\[
\widetilde{\mathcal O}\!\left(
\frac{L D_{\mathcal{X}}\sqrt d}{\sqrt T}
\;+\;
\frac{L D_{\mathcal{X}}\sqrt d}{(1-\sigma)\sqrt T}
\right).
\]
This yields the expected \(\varepsilon^{-2}\) complexity scaling, up to
logarithmic factors and $d$ dependence, for both stochastic zeroth-order oracle calls and
communication rounds. The experiments on decentralized classification problems
also indicate that D-POEM is robust without manual tuning and becomes
competitive with tuned distributed zeroth-order baselines after a moderate
number of iterations.

\paragraph{Discussion.}
Several questions remain open. On the theoretical side, lower bounds 
for decentralized zeroth-order convex optimization are still lacking,
and establishing them would help assess the efficiency of our method
in terms of both communication and oracle complexity. It would also 
be valuable to sharpen the dependence on the network 
connectivity factor $1-\sigma$, and derive fully high-probability 
guarantees for the final output. On the algorithmic side, our 
method requires two consensus rounds per iteration. Although one 
of them is used only for the radius, which is a scalar and 
therefore relatively cheap to communicate, eliminating this 
step would further reduce communication cost. The experiments 
also suggest that the basedline methods can still be improved, as 
carefully tuned baselines may converge faster in the early stage. 
Promising directions for future work include communication-efficient
 variants, extensions to time-varying or directed networks, 
 and broader parameter-free designs for weaker-feedback oracle 
 models or nonconvex objectives.

\begin{credits}
\subsubsection{\discintname}
 The authors have no competing interests to declare that are
relevant to the content of this article. 

\end{credits}

\appendix
\section{Proofs}\label{app:proofs}
\subsection{Auxilary Lemma}
\begin{lemma}[\cite{shamir2017optimal}, Lemma 10]\label{eq:norm-grad}
Under Assumption \ref{ass:lipschitz}, D-POEM (Algorithm \ref{alg:dpoem}) holds that
\[
\lVert g_t \rVert \le Ld
\qquad \text{and} \qquad
\mathbb{E}\!\left[\lVert g_t \rVert^2\right] \le cL^2 d,
\]
where \(c>0\) is a numerical constant.
\end{lemma}
\subsection{Decomposition and primary bounds}

\begin{lemma}[Basic decomposition]\label{lem:basic-decomposition}
For every \(t\ge 1\),
\[
R_t\bigl(f(\tilde x_t)-f(x^\star)\bigr)
\le
\mathcal W_t+\mathcal N_t+\mathcal B_t+\mathcal E_t.
\]
\end{lemma}

\begin{proof}
By convexity of \(f\) and the definition of \(\tilde x_t\),
\[
f(\tilde x_t)-f(x^\star)
\le
\frac{1}{R_t}\sum_{k=0}^{t-1}\bar r_k\bigl(f(\bar x_k)-f(x^\star)\bigr).
\]
For each \(k\),
\begin{align*}
\bar r_k\bigl(f(\bar x_k)-f(x^\star)\bigr)
&=
\bar r_k\cdot \frac{1}{n}\sum_{i=1}^{n}\bigl(f_i(\bar x_k)-f_i(x^\star)\bigr) \\
&\le
\bar r_k\cdot \frac{1}{n}\sum_{i=1}^{n}
\bigl(f_i(x_{i,k})-f_i(x^\star)+L\|x_{i,k}-\bar x_k\|\bigr) \\
&\le
\bar r_k\cdot \frac{1}{n}\sum_{i=1}^{n}
\bigl(f_{i,\mu_{i,k}}(x_{i,k})-f_{i,\mu_{i,k}}(x^\star)
+2L\mu_{i,k}+L\|x_{i,k}-\bar x_k\|\bigr).
\end{align*}
Using convexity of the function and writing
\(\Delta_{i,k}:=\nabla f_{i,\mu_{i,k}}(x_{i,k})-g_{i,k}\),
\[
f_{i,\mu_{i,k}}(x_{i,k})-f_{i,\mu_{i,k}}(x^\star)
\le
\langle g_{i,k},z_{i,k}-x^\star\rangle
+
\langle \Delta_{i,k},x_{i,k}-x^\star\rangle
+
\langle g_{i,k},x_{i,k}-z_{i,k}\rangle.
\]
Summing over \(i\) and \(k\) yields
\[
\sum_{k=0}^{t-1}\bar r_k\bigl(f(\bar x_k)-f(x^\star)\bigr)
\le
\mathcal W_t+\mathcal N_t+\mathcal B_t+\mathcal E_t.
\]
Dividing by \(R_t\) proves the claim.
\end{proof}

\begin{proof}[Proof of Lemma~\ref{lem:weighted-regret}]
Define
\[
d_{i,k}:=\|x_{i,k}-x^\star\|,
\qquad
\hat d_{i,t}:=\max_{0\le s\le t} d_{i,s}.
\]
By nonexpansiveness of projection,
\[
\|x_{i,k+1}-x^\star\|^2
=
\|\Pi_{\mathcal{X}}(z_{i,k}-\eta_{i,k}g_{i,k})-\Pi_{\mathcal{X}}(x^\star)\|^2
\le
\|z_{i,k}-\eta_{i,k}g_{i,k}-x^\star\|^2.
\]
Expanding and rearranging gives
\begin{equation}\label{eq:one-step-inner-appendix}
\langle g_{i,k},z_{i,k}-x^\star\rangle
\le
\frac{\|z_{i,k}-x^\star\|^2-\|x_{i,k+1}-x^\star\|^2}{2\eta_{i,k}}
+
\frac{\eta_{i,k}}{2}\|g_{i,k}\|^2.
\end{equation}
Multiplying by \(\bar r_k\) and inserting
\(\eta_{i,k}=\bar r_{i,k}/\sqrt{G_{i,k}}\) yields
\begin{align*}
\bar r_k\langle g_{i,k},z_{i,k}-x^\star\rangle
&\le
\frac{\sqrt{G_{i,k}}}{2}\bigl(\|z_{i,k}-x^\star\|^2-\|x_{i,k+1}-x^\star\|^2\bigr)
+\frac{\bar r_k\bar r_{i,k}}{2\sqrt{G_{i,k}}}\|g_{i,k}\|^2
+\mathcal M_{i,k},
\end{align*}
where
\[
\mathcal M_{i,k}
:=
\frac{\bar r_k-\bar r_{i,k}}{2\eta_{i,k}}
\bigl(\|z_{i,k}-x^\star\|^2-\|x_{i,k+1}-x^\star\|^2\bigr).
\]
Let \(\hat G_k:=\max_i G_{i,k}\). Then
\begin{align*}
\bar r_k\langle g_{i,k},z_{i,k}-x^\star\rangle
&\le
\frac{\sqrt{\hat G_k}}{2}\bigl(\|z_{i,k}-x^\star\|^2-\|x_{i,k+1}-x^\star\|^2\bigr)
+\frac{\bar r_k\bar r_{i,k}}{2\sqrt{G_{i,k}}}\|g_{i,k}\|^2
+\mathcal M_{i,k}.
\end{align*}
Summing over \(i\) and using Jensen's inequality,
$
\sum_{i=1}^{n}\|z_{i,k}-x^\star\|^2
\le
\sum_{i=1}^{n}\|x_{i,k}-x^\star\|^2.
$
Hence
\begin{align*}
\mathcal W_t
&\le
\frac{1}{2n}\sum_{k=0}^{t-1}\sum_{i=1}^{n}
\sqrt{\hat G_k}\bigl(d_{i,k}^2-d_{i,k+1}^2\bigr)
+
\frac{1}{2n}\sum_{k=0}^{t-1}\sum_{i=1}^{n}
\frac{\bar r_k\bar r_{i,k}}{\sqrt{G_{i,k}}}\|g_{i,k}\|^2
+
\frac{1}{n}\sum_{k=0}^{t-1}\sum_{i=1}^{n}\mathcal M_{i,k}.
\end{align*}
Because \(\sqrt{\hat G_k}\) is nondecreasing in \(k\), as \cite[Lemma 1]{ivgi2023dog}, then
\[
\frac{1}{2n}\sum_{k=0}^{t-1}\sum_{i=1}^{n}
\sqrt{\hat G_k}\bigl(d_{i,k}^2-d_{i,k+1}^2\bigr)
\le
\frac{2}{n}\sum_{i=1}^{n}\sqrt{\hat G_{t-1}}\hat r_{i,t}\hat d_{i,t}
\le
2D_{\mathcal{X}}\bar r_t\sqrt{\hat G_{t-1}}.
\]
Also,
\[
\frac{1}{2n}\sum_{k=0}^{t-1}\sum_{i=1}^{n}
\frac{\bar r_k\bar r_{i,k}}{\sqrt{G_{i,k}}}\|g_{i,k}\|^2
\le
\frac{D_{\mathcal{X}}\bar r_t}{n}\sum_{i=1}^{n}\sqrt{G_{i,t-1}}
\le
D_{\mathcal{X}}\bar r_t\sqrt{\hat G_{t-1}}.
\]
Finally, since \(\sum_{i=1}^{n}(\bar r_k-\bar r_{i,k})=0\) for every \(k\),
$
\sum_{i=1}^{n}\mathcal M_{i,k}=0.
$
Combining the three displays yields
$
\mathcal W_t\le 3D_{\mathcal{X}}\bar r_t\sqrt{\hat G_{t-1}}.
$
\end{proof}

\begin{proof}[Proof of Lemma~\ref{lem:dec-noise}]
Let \(\mathcal{F}_k\) be the filtration generated by all oracle randomness up
to time \(k\), and define
\[
d_{i,k}:=\|x_{i,k}-x^\star\|,
\qquad
\hat d_k:=\max_{j\le k}\max_{i\le n} d_{i,j}.
\]
This is the same argument as \cite[Lemma~5]{renparameter}, applied to the averaged
noise process
\[
X_k
:=
\frac{1}{\hat d_k}\cdot \frac1n\sum_{i=1}^n \langle \Delta_{i,k},x_{i,k}-x^\star\rangle,
\qquad
\widehat X_k
:=
\frac{1}{\hat d_k}\cdot \frac1n\sum_{i=1}^n
\langle \nabla f_{i,\mu_{i,k}}(x_{i,k}),x_{i,k}-x^\star\rangle.
\]
Because \(x_{i,k}\) and \(\hat d_k\) are \(\mathcal F_{k-1}\)-measurable and
\(\mathbb E[g_{i,k}\mid\mathcal F_{k-1}]
=\nabla f_{i,\mu_{i,k}}(x_{i,k})\), the sequence
\((X_k,\mathcal F_k)\) is a martingale difference sequence, and from Lemma \ref{eq:norm-grad}
we can know \(\|\Delta_{i,k}\|\le 2Ld\), it implies \(|X_k|,|\widehat X_k|\le 2Ld\) almost
surely. Moreover,
\[
X_k-\widehat X_k
=
-\frac{1}{\hat d_k}\cdot \frac1n\sum_{i=1}^n \langle g_{i,k},x_{i,k}-x^\star\rangle,
\]
so
\[
\sum_{k=0}^{t-1}(X_k-\widehat X_k)^2
\le
\frac{1}{n}\sum_{i=1}^n \sum_{k=0}^{t-1}\|g_{i,k}\|^2
\le
\hat G_{t-1}.
\]
Therefore \cite[Lemma~5]{renparameter}, with predictable weights
\(Y_k:=\bar r_k\hat d_k\), gives an event \(\Omega_{T,\delta}\subseteq \Omega\) with
\(\Prob(\Omega_{T,\delta})\ge 1-\delta\) such that, simultaneously for all
\(1\le t\le T\),
\[
\sum_{k=0}^{t-1}\bar r_k\cdot \frac1n\sum_{i=1}^n
\langle \Delta_{i,k},x_{i,k}-x^\star\rangle
\le
8\bar r_{t-1}\hat d_{t-1}
\sqrt{\theta_{t,\delta}\hat G_{t-1}+4L^2d^2\theta_{t,\delta}^2}.
\]
Using \(\hat d_{t-1}\le D_{\mathcal{X}}\) proves the lemma.
\end{proof}

\begin{proof}[Proof of Lemma~\ref{lem:smoothing-bias}]
Since \(\mu_{i,k}=\bar r_{i,k}\sqrt{d/(k+1)}\),
\[
\frac{L}{n}\sum_{i=1}^{n}\mu_{i,k}
\le
L\bar r_k\sqrt{\frac{d}{k+1}}.
\]
Therefore
$\mathcal B_t
\le
2L\sqrt d\sum_{k=0}^{t-1}\frac{\bar r_k^2}{\sqrt{k+1}} 
\le
2L D_{\mathcal{X}}\bar r_t\sqrt d\sum_{k=0}^{t-1}\frac{1}{\sqrt{k+1}} 
\le
4D_{\mathcal{X}}L\bar r_t\sqrt{dt},
$
where the last step uses
\(\sum_{k=0}^{t-1}(k+1)^{-1/2}\le 2\sqrt t\).
\end{proof}

\subsection{Consensus error}
\begin{proof}[Proof of Lemma~\ref{lem:avg-consensus-A}]
Let \(S_k:=\mathrm{diag}(\eta_{1,k},\dots,\eta_{n,k})\). For D-POEM,
\[
X_{k+1}
=
\Pi_{\mathcal{X}}^{\mathrm{row}}\!\left(WX_k-S_k\mathbf G_k\right).
\]
Write \(\widetilde X_k:=(I-J)X_k\) and \(E_k:=\|\widetilde X_k\|_F\).
By row-wise nonexpansiveness of projection and \((I-J)W=(W-J)\),
\[
\|(I-J)X_{k+1}\|_F
\le
\|(W-J)X_k\|_F+\|(I-J)S_k\mathbf G_k\|_F.
\]
Hence, with \(U_k:=\|(I-J)S_k\mathbf G_k\|_F\),
\[
E_{k+1}\le \sigma E_k+U_k
\quad\Longrightarrow\quad
(1-\sigma)\sum_{k=0}^{K-1}E_k\le E_0+\sum_{k=0}^{K-1}U_k.
\]
Moreover,
\[
U_k^2
\le
\|S_k\mathbf G_k\|_F^2
=
\sum_{i=1}^{n}\eta_{i,k}^2\|g_{i,k}\|^2
=
\sum_{i=1}^{n}\bar r_{i,k}^2\frac{\|g_{i,k}\|^2}{G_{i,k}}
\le
D_{\mathcal{X}}^2\sum_{i=1}^{n}\frac{\|g_{i,k}\|^2}{G_{i,k}}.
\]
Using the elementary inequality
\[
\sum_{k=0}^{K-1}\frac{a_k}{A_k}\le \log\!\Bigl(\frac{A_{K-1}}{A_0}\Bigr),
\qquad
A_k:=A_{k-1}+a_k,
\]
with \(a_k=\|g_{i,k}\|^2\) and \(A_k=G_{i,k}\), we obtain
\[
\sum_{k=0}^{K-1}U_k^2
\le
D_{\mathcal{X}}^2\sum_{i=1}^{n}\log\!\Bigl(\frac{G_{i,K-1}}{G_{i,0}}\Bigr).
\]
Cauchy--Schwarz then yields
\[
\sum_{k=0}^{K-1}U_k
\le
D_{\mathcal{X}}\sqrt{K}
\Biggl(\sum_{i=1}^{n}\log\!\Bigl(\frac{G_{i,K-1}}{G_{i,0}}\Bigr)\Biggr)^{1/2}.
\]
Substituting into the earlier recursion and dividing by \(K(1-\sigma)\)
proves the lemma.
\end{proof}

\begin{proof}[Proof of Lemma~\ref{lem:consensus-exp}]
Let \(X_k\) stack the iterates row-wise, let \(Z_k:=WX_k\), and define
\[
\widetilde X_k:=(I-J)X_k,
\qquad
E_k:=\|\widetilde X_k\|_F,
\qquad
\mathbf G_k:=\begin{bmatrix} g_{1,k}^\top \\ \vdots \\ g_{n,k}^\top \end{bmatrix}.
\]
For the mixing term in \(\mathcal E_t\),
\[
\sum_{k=0}^{t-1}\frac{\bar r_k}{n}\sum_{i=1}^n
\langle g_{i,k},x_{i,k}-z_{i,k}\rangle
\le
\frac{1}{n}\sum_{k=0}^{t-1}\bar r_k\,\|\mathbf G_k\|_F\,\|X_k-Z_k\|_F.
\]
Since \(X_k-Z_k=(I-W)(I-J)X_k\), we have
\[
\|X_k-Z_k\|_F\le 2E_k,
\]
and therefore
\[
\sum_{k=0}^{t-1}\frac{\bar r_k}{n}\sum_{i=1}^n
\langle g_{i,k},x_{i,k}-z_{i,k}\rangle
\le
\frac{2}{n}\sum_{k=0}^{t-1}\bar r_k\,\|\mathbf G_k\|_F\,E_k.
\]
For the state-disagreement term,
\[
\sum_{k=0}^{t-1}\frac{L\bar r_k}{n}\sum_{i=1}^n \|x_{i,k}-\bar x_k\|
\le
\frac{L}{\sqrt n}\sum_{k=0}^{t-1}\bar r_k E_k.
\]
Hence add both of them, we conclude this lemma.
\end{proof}

\subsection{Proof of the main theorems}

\begin{proof}[Proof of Theorem~\ref{thm:function-gap}]
Lemma~\ref{lem:basic-decomposition} gives
\(
R_t\bigl(f(\tilde x_t)-f(x^\star)\bigr)
\le
\mathcal W_t+\mathcal N_t+\mathcal B_t+\mathcal E_t.
\)
On the event \(\Omega_{T,\delta}\) from Lemma~\ref{lem:dec-noise}, the bounds
from Lemmas~\ref{lem:weighted-regret}, \ref{lem:dec-noise},
\ref{lem:smoothing-bias}, and \ref{lem:consensus-exp} imply
\begin{align*}
R_t\bigl(f(\tilde x_t)-f(x^\star)\bigr)
&\le
3D_{\mathcal{X}}\bar r_t\sqrt{\hat G_{t-1}}
+
8D_{\mathcal{X}}\bar r_{t-1}
\sqrt{\theta_{t,\delta}\hat G_{t-1}+4L^2d^2\theta_{t,\delta}^2} \\
&\quad+
4D_{\mathcal{X}}L\bar r_t\sqrt{dt}
+
\mathcal E_t.
\end{align*}
Since \(t\le T\), \(\bar r_{t-1}\le \bar r_t\), \(\hat G_{t-1}\le \hat G_T\),
\(\theta_{t,\delta}\le \theta_{T,\delta}\),
we obtain
\begin{align*}
R_t\bigl(f(\tilde x_t)-f(x^\star)\bigr)
&\le
\bar r_t\Bigl(
3D_{\mathcal{X}}\sqrt{\hat G_T}
+
8D_{\mathcal{X}}
\sqrt{\theta_{T,\delta}\hat G_T+4L^2d^2\theta_{T,\delta}^2}
+
4D_{\mathcal{X}}L\sqrt{dT}
\Bigr) \\
&\quad+
\mathcal E_t.
\end{align*}
Next, \(\sqrt{a+b}\le \sqrt a+\sqrt b\) gives
\(
\sqrt{\theta_{T,\delta}\hat G_T+4L^2d^2\theta_{T,\delta}^2}
\le
\sqrt{\theta_{T,\delta}}\sqrt{\hat G_T}+2Ld\,\theta_{T,\delta}.
\)
Since \(\theta_{T,\delta}=\log(60\log(6T/\delta))\ge 1\), we have
\(
3+8\sqrt{\theta_{T,\delta}}\le 16\theta_{T,\delta}.
\)
Therefore the three optimization terms can be collected into a single envelope
of the form
\[
16D_{\mathcal{X}}\theta_{T,\delta}\bar r_t
\bigl(\sqrt{\hat G_T}+Ld+L\sqrt{dT}\bigr),
\]
up to harmless absolute-constant slack. Dividing by
\(
R_t=\sum_{k=0}^{t-1}\bar r_k
\)
and rewriting
\(
R_t/\bar r_t=\sum_{k=0}^{t-1}\bar r_k/\bar r_t
\)
yields \eqref{eq:function-gap-divided}.
\end{proof}

\subsection{Expected-rate proof}\label{app:proof-expected-rate}

\begin{proposition}[Explicit conditional expected rate]\label{prop:expected-rate-explicit}
Under the assumptions of Theorem~\ref{thm:expected-rate},

\begin{align}
\label{eq:expected-rate-main}
\mathbb E&\!\left[f(\bar x_{\mathrm{out}})-f(x^\star)\mid \Omega_{T,\delta}\right]\\
&\le
\frac{16e\bigl(\log(\frac{D_{\mathcal{X}}}{r_\epsilon})_{+}\bigr)D_{\mathcal{X}}\theta_{T,\delta}}{T}
\left(
\frac{\sqrt{r_\epsilon^2+TcL^2d}}{\sqrt{1-\delta}}
+
L(d+\sqrt{dT})
\right)
\nonumber\\
&\quad+
\frac{e\bigl(\log(\frac{D_{\mathcal{X}}}{r_\epsilon})_{+}\bigr)(2\sqrt{cd}+1)L D_{\mathcal{X}}}
{(1-\delta)(1-\sigma)\sqrt{T}}
\sqrt{\log\!\Bigl(1+\frac{TcL^2d}{\min_i G_{i,0}}\Bigr)}.
\end{align}

\end{proposition}

\begin{proof}
Let
\(
\Lambda_T:=\frac{e\bigl(1+\log(D_{\mathcal{X}}/r_\epsilon)\bigr)}{T}.
\)
Since \((\bar r_t)_{t\ge 0}\) is positive and nondecreasing,
the same monotone-sequence bound used in the  POEM  gives
\(
\frac{R_{\tau_T}}{\bar r_{\tau_T}}
\ge
\frac{1}{\Lambda_T}.
\)
Applying Theorem~\ref{thm:function-gap} at \(t=\tau_T\), using
\(\tau_T\le T\) and monotonicity of \(t\mapsto \theta_{t,\delta}\), yields on
\(\Omega_{T,\delta}\), holds
\(
f(\bar x_{\mathrm{out}})-f(x^\star)
\le
16D_{\mathcal{X}}\theta_{T,\delta}\Lambda_T
\bigl(\sqrt{\hat G_{T}}+Ld+L\sqrt{dT}\bigr)
\;+\;
\frac{\mathcal E_{\tau_T}}{R_{\tau_T}}.
\)
Taking conditional expectations, we first control the network maximum
accumulator. Since
\(
\hat G_T
\le
r_\epsilon^2+\sum_{k=0}^{T-1}\max_{i\in[n]}\|g_{i,k}\|^2,
\)
the conditional max-moment assumption in Theorem~\ref{thm:expected-rate} gives
\(
\mathbb E[\hat G_T]
\le
r_\epsilon^2+TcL^2d.
\)
Therefore, by Jensen's inequality and \(\Prob(\Omega_{T,\delta})\ge 1-\delta\),
\[
\mathbb E[\sqrt{\hat G_{T}}\mid \Omega_{T,\delta}]
\le
\sqrt{\mathbb E[\hat G_T\mid \Omega_{T,\delta}]}
\le
\frac{\sqrt{r_\epsilon^2+TcL^2d}}{\sqrt{1-\delta}}.
\]
\[
\mathbb E\!\left[\frac{\mathcal E_{\tau_T}}{R_{\tau_T}}\middle|\Omega_{T,\delta}\right]
\le
\frac{1}{1-\delta}\,
\mathbb E\!\left[\frac{\mathcal E_{\tau_T}}{R_{\tau_T}}\right],
\]
because \(\mathcal E_{\tau_T}\ge 0\) and \(\Prob(\Omega_{T,\delta})\ge 1-\delta\).
It remains to upper bound the normalized consensus term on the right-hand side.
By Lemma~\ref{lem:consensus-exp},
\[
\frac{\mathcal E_{\tau_T}}{R_{\tau_T}}
\le
\frac{2}{n}\sum_{k=0}^{\tau_T-1}\frac{\bar r_k}{R_{\tau_T}}\|\mathbf G_k\|_F E_k
+
\frac{L}{\sqrt n}\sum_{k=0}^{\tau_T-1}\frac{\bar r_k}{R_{\tau_T}}E_k.
\]
Since \(\tau_T\) maximizes \(R_t/\bar r_t\) and \((\bar r_t)\) is nondecreasing,
for every \(k\le \tau_T-1\),
\[
\frac{\bar r_k}{R_{\tau_T}}
\le
\frac{\bar r_{\tau_T}}{R_{\tau_T}}
\le
\Lambda_T.
\]
Therefore,
\(
\frac{\mathcal E_{\tau_T}}{R_{\tau_T}}
\le
\frac{2\Lambda_T}{n}\sum_{k=0}^{T-1}\|\mathbf G_k\|_F E_k
+
\frac{L\Lambda_T}{\sqrt n}\sum_{k=0}^{T-1}E_k.
\)
Taking expectations and using that \(E_k\) is \(\mathcal F_{k-1}\)-measurable,
\begin{align*}
\mathbb E\!\left[\frac{\mathcal E_{\tau_T}}{R_{\tau_T}}\right]
&\le
\frac{2\Lambda_T}{n}\sum_{k=0}^{T-1}
\mathbb E\!\left[E_k\,\mathbb E[\|\mathbf G_k\|_F\mid \mathcal F_{k-1}]\right]
+
\frac{L\Lambda_T}{\sqrt n}\sum_{k=0}^{T-1}\mathbb E[E_k] \\
&\le
\frac{2\Lambda_T}{n}\sum_{k=0}^{T-1}
\mathbb E\!\left[E_k\Bigl(\mathbb E[\|\mathbf G_k\|_F^2\mid \mathcal F_{k-1}]\Bigr)^{1/2}\right]
+
\frac{L\Lambda_T}{\sqrt n}\sum_{k=0}^{T-1}\mathbb E[E_k].
\end{align*}
Since we can know expectation of  gradient norm from the Lemma \ref{eq:norm-grad}, then 
\(
\mathbb E[\|\mathbf G_k\|_F^2\mid \mathcal F_{k-1}]
=
\sum_{i=1}^n \mathbb E[\|g_{i,k}\|^2\mid \mathcal F_{k-1}]
\le
ncL^2d,
\)
where the last inequality uses
\(\sum_{i=1}^n \|g_{i,k}\|^2\le n\max_{i\in[n]}\|g_{i,k}\|^2\) together with the Lemma \ref{eq:norm-grad}.
so
\(
\mathbb E\!\left[\frac{\mathcal E_{\tau_T}}{R_{\tau_T}}\right]
\le
\frac{(2\sqrt{cd}+1)L\Lambda_T}{\sqrt n}\sum_{k=0}^{T-1}\mathbb E[E_k].
\)
Next we estimate the expected cumulative disagreement. Lemma~\ref{lem:avg-consensus-A}
with \(K=T\) gives
\(
\sum_{k=0}^{T-1}E_k
\le
\frac{D_{\mathcal{X}}T}{1-\sigma}
\sqrt{\frac{1}{T}\sum_{i=1}^{n}\log\!\Bigl(\frac{G_{i,T-1}}{G_{i,0}}\Bigr)}.
\)
Taking expectations and applying Jensen's inequality twice,
\begin{align*}
\sum_{k=0}^{T-1}\mathbb E[E_k]
&\le
\frac{D_{\mathcal{X}}T}{1-\sigma}
\sqrt{\frac{1}{T}\sum_{i=1}^{n}
\mathbb E\!\left[\log\!\Bigl(\frac{G_{i,T-1}}{G_{i,0}}\Bigr)\right]} \le
\frac{D_{\mathcal{X}}T}{1-\sigma}
\sqrt{\frac{1}{T}\sum_{i=1}^{n}
\log\!\Bigl(\frac{\mathbb E[G_{i,T-1}]}{G_{i,0}}\Bigr)}.
\end{align*}
Finally,
\(
\mathbb E[G_{i,T-1}]
=
G_{i,0}+\sum_{k=0}^{T-1}\mathbb E[\|g_{i,k}\|^2]
\le
G_{i,0}+TcL^2d,
\)
hence
\(
\sum_{k=0}^{T-1}\mathbb E[E_k]
\le
\frac{D_{\mathcal{X}}\sqrt{nT}}{1-\sigma}
\sqrt{\log\!\Bigl(1+\frac{TcL^2d}{\min_i G_{i,0}}\Bigr)}.
\)
Combining the last three displays yields
\[
\mathbb E\!\left[\frac{\mathcal E_{\tau_T}}{R_{\tau_T}}\middle|\Omega_{T,\delta}\right]
\le
\frac{e\bigl(1+\log(D_{\mathcal{X}}/r_\epsilon)\bigr)(2\sqrt{cd}+1)L D_{\mathcal{X}}}
{(1-\delta)(1-\sigma)\sqrt T}
\sqrt{\log\!\Bigl(1+\frac{TcL^2d}{\min_i G_{i,0}}\Bigr)}.
\]
Substituting this bound and the estimate
\(
\mathbb E[\sqrt{\hat G_T}\mid \Omega_{T,\delta}]
\le
\frac{\sqrt{r_\epsilon^2+TcL^2d}}{\sqrt{1-\delta}}
\)
into the earlier conditional bound proves \eqref{eq:expected-rate-main}.
\end{proof}

\begin{proof}[Proof of Theorem~\ref{thm:expected-rate}]
The previous proposition is the explicit version of Theorem~\ref{thm:expected-rate}.
Indeed, Proposition~\ref{prop:expected-rate-explicit} already gives a conditional
bound with two leading scales:
\[
\frac{L D_{\mathcal{X}}\sqrt d}{\sqrt T}
\qquad\text{and}\qquad
\frac{L D_{\mathcal{X}}\sqrt d}{(1-\sigma)\sqrt T}.
\]
The extra factors
\(
\theta_{T,\delta},
\log\!\bigl(1+TcL^2d/\min_i G_{i,0}\bigr),
\log(D_{\mathcal{X}}/r_\epsilon)
\)
are logarithmic in the horizon and are therefore hidden by
\(\widetilde{\mathcal O}(\cdot)\). The remaining \(T^{-1}\) terms, including
\(L D_{\mathcal{X}}d/T\) and the contribution of \(r_\epsilon\), are lower
order for fixed \(d\). Hence \eqref{eq:expected-rate-tilde} follows.
\end{proof}

\bibliographystyle{splncs04}
\bibliography{mybibliography}
%




\end{document}